\documentclass[12pt]{amsart}
\usepackage[dvips]{graphicx,color}
\usepackage{amsmath,amssymb,amsthm,mathrsfs}
\usepackage[vcentermath]{youngtab}
\usepackage{amscd}
       
\usepackage{a4wide}
\usepackage{euscript}
\usepackage{amsthm}
\usepackage{amsopn}
\usepackage{helvet}
\usepackage{color}
\usepackage[all]{xy}

\theoremstyle{definition}

\newtheorem{theorem}{Theorem}
\newtheorem{lemma}[theorem]{Lemma}
\newtheorem{prop}[theorem]{Proposition}
\newtheorem{example}[theorem]{Example}
\newtheorem{corollary}[theorem]{Corollary}

\newcommand{\sP}{\mathcal{P}}
\newcommand{\sQ}{\mathcal{Q}}

\begin{document}

\title{Properties of Generalized Derangement Graphs}
\author{Hannah Jackson}
\address{Mathematics Department, 
Syracuse University, 
215 Carnegie, 
Syracuse, New York 13244 U.S.A.}
\email{hljackso@syr.edu}

\author{Kathryn Nyman}
\address{Mathematics Department, 
Willamette University, 900 State Street, 
Salem, Oregon 97301 U.S.A.}
\email{knyman@willamette.edu}

\author{Les Reid}
\address{Mathematics Department,
Missouri State University, 901 South National Avenue,
Springfield, Missouri 65897 U.S.A.}
\email{LesReid@MissouriState.edu}

\subjclass[2000]{05C69, 05A05, 05C45} 
\date{\today}
\keywords{Derangements,  Eulerian, Chromatic Number, Maximal Clique, Cayley Graph, Independent Set}

\begin{abstract}
A permutation $\sigma \in S_n$ is a $k$-derangement if for any subset $X=\{a_1,\ldots,a_k\} \subseteq [n]$, $\{\sigma(a_1),\ldots,\sigma(a_k)\} \neq X$.  One can form the $k$-derangement graph on the set of permutations of $S_n$ by connecting two permutations $\sigma$ and $\tau$ if $\sigma \tau^{-1}$ is a $k$-derangement.  We characterize when such a graph is connected or Eulerian.  For  $n$ an odd prime power, we determine the independence, clique and chromatic number of the 2-derangement graph.  

\end{abstract}

\maketitle

\section{ Introduction}

Permutations which leave no element fixed, known as derangements,
were first considered by Pierre Raymond de Montmort in 1708 and have been 
extensively studied since. A derangement graph is a graph whose vertices are 
the elements of the symmetric group $S_n$ and whose edges connect two 
permutations that differ by a derangement. Derangement graphs have been shown
to be connected (for $n>3$), Hamiltonian, and their independence number, 
clique number, and chromatic number have been calculated \cite{R}.

The concept of a derangement can be generalized to a $k$-derangement, a
permutation in $S_n$ such that the induced permutation on the set of all
unordered $k$-tuples leaves no $k$-tuple fixed. A $k$-derangement graph is defined
in an analogous manner to a derangement graph. In this paper, we investigate
some of the graph-theoretical properties of $k$-derangement graphs.

\section{Preliminaries}

Let $S_n$ be the group of permutations on the set $[n] = \{1, 2, \ldots, n\}$, and
denote by $[n]^{(k)}$ the set of unordered $k$-tuples
with entries from $[n]$.  
Note that a permutation $\sigma \in S_n$ induces a
permutation $\sigma_{(k)}$ of unordered $k$-tuples by
$\sigma_{(k)}(\{a_1,\ldots,a_k\}) = \{\sigma(a_1),\ldots,\sigma(a_k)\}$.
For example, with $n=4$, $k=2$, and $\sigma = (1234)$ in cycle notation, we have

$$(1234)_{(2)}(\{1,2\})=\{(1234)(1),(1234)(2)\}=\{2,3\}$$
$$(1234)_{(2)}(\{1,3\})=\{(1234)(1),(1234)(3)\}=\{2,4\}$$
$$(1234)_{(2)}(\{1,4\})=\{(1234)(1),(1234)(4)\}=\{2,1\}=\{1,2\}$$
$$(1234)_{(2)}(\{2,3\})=\{(1234)(2),(1234)(3)\}=\{3,4\}$$
$$(1234)_{(2)}(\{2,4\})=\{(1234)(2),(1234)(4)\}=\{3,1\}=\{1,3\}$$
$$(1234)_{(2)}(\{3,4\})=\{(1234)(3),(1234)(4)\}=\{4,1\}=\{1,4\}.$$

Let
$\mathcal{D}_n := \{\sigma \in S_n | \sigma(x) \neq x, \forall x \in
$[$n$]$ \}$ denote the $\emph{ordinary derangements}$ on [$n$].
Extending this concept, we say that a permutation $\sigma \in S_n$
is a $k$$\emph{-derangement}$ if $\sigma_{(k)}(x) \neq x$ for all $x
\in $[$n$]$^{(k)}$. In other words, a $k$-derangement in $S_n$ is a
permutation (of [$n$]) which induces a permutation (of
[$n$]$^{(k)}$) which leaves no $k$-tuple fixed. The set of
$k$-derangements in $S_n$ is denoted $\mathcal{D}_{k,n}$, and the
number of $k$-derangements in $S_n$ is denoted $D_k(n)$ ($D_k(n) =
|\mathcal{D}_{k,n}|$).  
The example above shows that $(1234)$ is in $\mathcal{D}_{2,4}$.  Specifically, 
 $\mathcal{D}_{2,4} = \{(1234),(1243),(1324),(1342),(1423),(1432),(123)(4),(124)(3)(132)(4),(134)(2),\\
 (142)(3),(143)(2),(234)(1),(243)(1)
 \}$, and thus
$D_2(4)=14$.
Note that $\mathcal{D}_n = \mathcal{D}_{1,n}$, and
$D_1(n)$ is the ordinary derangement number.

The
cycle structure of a permutation $\sigma$, denoted $C_{\sigma}$, is
the multiset of the lengths of the cycles in its cycle
decomposition (e.g., $C_{(12)(3)(45)} = \{2,2,1\}$).
 Note that the cycle structure of
$\sigma \in S_n$ is a partition of $n$.  Given a partition $r \vdash n$, let $P_r$ be the set of all
permutations in $S_n$ whose cycle structure is $r$.  For example, 
$P_{\{2,1,1\}} = \{(12),(13),(14),(23),(24),(34)\}$.

We first note that if the cycle structure of a permutation $\sigma$ contains a multiset which partitions $k$, then $\sigma$ is not a $k$-derangement.   For
example, $(12)(34)$ will be a $3$-derangement in $S_4$, but
$(12)(3)(4)$ will not be, because $\{2,1\} \subseteq C_{(12)(3)(4)}
= \{2,1,1\}$ is a partition of 3.
Indeed, if $\{q,r,\ldots,s\}$ is a partition of $k$, and $(a_1\ldots a_q)(b_1 \ldots b_r)\ldots (c_1 \ldots c_s)$ are cycles of $\sigma$, then for $x = \{a_1,\ldots,a_q,b_1, \ldots,b_r,  c_1, \ldots, c_s\}, \sigma_{(k)}(x) = x$.  
Conversely, if $\sigma$ has no set of cycles whose lengths partition $k$, then given any $x \in [n]^{(k)}$, there is a cycle in $\sigma$ which contains at least one element in $x$ and contains some element not in $x$. Hence $\sigma$ sends an element in $x$ to an element not in $x$ and so $\sigma_{(k)}(x) \neq x$.  

Thus we observe that the cycle structure of a permutation determines whether or not it is a $k$-derangement, and we have the following.
\begin{prop}
A permutation $\sigma \in S_n$ is a $k$-derangement if and only if the cycle decomposition of $\sigma$ does not contain a set of cycles whose lengths partition $k$.
\end{prop}

Let $CD_{ k,n}$ be the set of cycle structures corresponding to $k$-derangements in $S_n$ [e.g., $CD_{ 2,4} =
\{\{4\},\{3,1\}\}$],
Note that $\mathcal{D}_{k,n} = \mathcal{D}_{(n-k),n}$. This follows
from the fact that if a cycle structure $C_\sigma$ is in $CD_{ k,n}$, then $C_\sigma$ is in  $CD_{ (n-k),n}$ as well.  

Let $G$ be a group, and let $S \subseteq G$ such that if $s$ is in $S$, then
$ s^{-1} $ is in $S$. The {$\emph{Cayley graph}$ $\Gamma(G,S)$
is the graph whose vertices are the elements of $G$ such that an edge connects
two vertices $u,v \in G$ if $su = v$ for some $s \in
S$. A $k$-{$\emph{derangement graph}$ is a Cayley graph defined by
$\Gamma_{k,n} := \Gamma(S_n, \mathcal{D}_{k,n})$. (Note that
$\mathcal{D}_{k,n}$ is symmetric, as the inverse of a
$k$-derangement is a $k$-derangement, and thus satisfies the
requirements for a Cayley graph.)  It is worth noting that $\Gamma_{k,n}$ is, by construction,
$D_k(n)$-regular, and that since $\mathcal{D}_{k,n} =
\mathcal{D}_{(n-k),n}$, $\Gamma_{k,n} = \Gamma_{(n-k),n}$.
Figure \ref{gamma(2,3)} illustrates the 2-derangement
graph on 6 vertices, $\Gamma_{2,3}$.

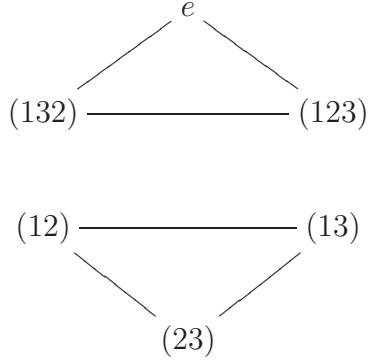
\begin{figure}
\begin{center}
$\displaystyle{ \xymatrix{
& e \ar@{-}[dl] \ar@{-}[dr] & \\
(132) \ar@{-}[rr] && (123) \\
(12) \ar@{-}[rr] && (13) \\
& (23) \ar@{-}[ul] \ar@{-}[ur] }}$
\end{center}
\caption{ $\Gamma_{2,3}$}
\label{gamma(2,3)}
\end{figure}

 It is possible to consider $k$-derangements in $S_n$ for
any positive $k$ and $n$. However, if $k=n$, there will be no $k$-derangements in
$S_n$, since every partition in $S_n$ will have a cycle structure
such that the cycle lengths partition $k$. As such, $\Gamma_{k,n}$
will be the empty graph on $n$ vertices. If $k > n$, then every
permutation in $S_n$ is a $k$-derangement vacuously, and thus
$\Gamma_{k,n}$ will be the complete graph on $|S_n|$ vertices.
As neither of these cases is particularly interesting,
henceforth we will only consider $k$-derangements where $k < n$.

\section{Properties of derangement graphs}

Figure 1 shows that $\Gamma_{2,3}$ is not a connected
graph, and since $\Gamma_{2,3} =
\Gamma_{1,3}$, we see that $\Gamma_{k,3}$ is disconnected, for all $k <n$.  But this is an exception rather than the rule, as the following theorem demonstrates.

\begin{theorem} For $n>3$ and $k<n$, $\Gamma_{k,n}$ is connected.
\label{thm:connected}

\begin{proof} Every permutation in $S_n$ can be written as the product of adjacent transpositions $(h \hspace{.05in}(h+1))$. These, in turn, can be expressed as the product of
two $k$-derangements, so long as $n > 3$, as we will demonstrate.   As a result, for $n > 3$, the elements of $\mathcal{D}_{k,n}$
generate $S_n$, which means that every vertex of $\Gamma_{k,n}$ can
be reached by a path from the identity. 

We show that the permutation $(1\hspace{.05in}2)$ can be written as the
product of two $k$-derangements and then note that since it is the form and not the individual labels that are important, any adjacent transposition can be written as the product of two $k$-derangements.
We consider two cases, the case where $k=1$, and the
case where $k \geq 2$.  

\vspace{.25cm}

\noindent {Case 1:} If $k=1$, then $(1\hspace{.05in}2) = (1\hspace{.05in}2 \ldots n)^2 \cdot
(n  \hspace{.05in}(n-1) \ldots 1)^2(1\hspace{.05in}2)$. We claim that $(1\hspace{.05in} 2 \ldots n)^2$ and
$(n \hspace{.05in}(n-1) \ldots 1)^2(1 \hspace{.05in}2)$ are each 1-derangements in $S_n$ for all $n >
3$. 
If $n$ is even, then
$(1\hspace{.05in}2 \ldots n)^2 = (1\hspace{.05in} 3 \ldots (n-3) \hspace{.05in}(n-1))(2\hspace{.05in} 4 \ldots (n-2) \hspace{.05in}n)$, \indent which is
a $1$-derangement in $S_n$, for all $n$. Additionally,
$(n\hspace{.05in}(n-1)\ldots 1)^2(1\hspace{.05in}2) =
(1\hspace{.05in}n\hspace{.05in}(n-2)\hspace{.05in}(n-4)\ldots 2\hspace{.05in}(n-1)\hspace{.05in}(n-3)\ldots 3),$ which is also a
$1$-derangement in $S_n$, for any $n$.

On the other hand, if $n$ is odd, then
$(1\hspace{.05in}2\hspace{.05in}\ldots n)^2 = (1\hspace{.05in}3\ldots (n-2)\hspace{.05in}n\hspace{.05in}2\hspace{.05in}4 \ldots (n-3)\hspace{.05in}(n-1))$, which is
a $1$-derangement in $S_n$ for all $n$. And
$(n\hspace{.05in}(n-1)\ldots1)^2(1\hspace{.05in}2) = 
(n\hspace{.05in}(n-2) \hspace{.05in}(n-4) \ldots 3\hspace{.05in}1\hspace{.05in}(n-1)\hspace{.05in}(n-3)\ldots 4\hspace{.05in}2)(1\hspace{.05in}2) =
(1\hspace{.05in}n\hspace{.05in}(n-2)\hspace{.05in}(n-4)\ldots 3)(2\hspace{.05in}(n-1)\hspace{.05in}(n-3)\ldots 4)$, which is a
$1$-derangement in $S_n$ so long as $n > 3$. (If $n = 3$, 
$(312)(12) = (13)(2)$, which is not a 1-derangement.)

Thus for $n > 3$, 
we have shown that
$(1\hspace{.05in}2)$ can be written as the product of two $1$-derangements, and,
by extension, every adjacent
transposition can be written as the product of two 1-derangments.

\vspace{.25cm}
\noindent {Case 2:} For $k \geq 2$, $(1\hspace{.05in}2) =
(1\hspace{.05in}2\ldots n)^{-1}(1\hspace{.05in}3\hspace{.05in}4\ldots n)$.  We know $(1\hspace{.05in}2\ldots n)^{-1}$ is a
$k$-derangement for all $k$ since the inverse of a $k$-derangement
is a $k$-derangement.  And, by the cycle structure, we see that $(1\hspace{.05in}3\hspace{.05in}4\ldots n) = (1\hspace{.05in}3\hspace{.05in}4\ldots n)(2)$ is a
$k$-derangement for all $k$, except $k=1$ and $k=(n-1)$(However, since $\Gamma_{1,n} = \Gamma_{(n-1),n}$, Case 1 addresses
$(n-1)$-derangements as well as $1$-derangements). 

So we have shown that for $k \geq 2$, $(1\hspace{.05in}2)$ can be written as
the product of two $k$-derangements, and again, by extension,
we can write any adjacent transposition
as the product of two $k$-derangements.  Thus every vertex is connected by a path to the identity, and $\Gamma_{k,n}$ is connected.

\end{proof}
 \end{theorem}

It is worth noting that Theorem 1 holds for $n = 2$ as well. Since we are only interested in
$k$-derangements in $S_n$ such that $k < n$, when $n=2$, $k$ must equal 1, and so
$\Gamma_{1,2}$ is the connected graph on two vertices.

Next, we give a characterization in terms of $n$ and $k$ for when a derangement graph is Eulerian.  We will require the following result.

\begin{lemma} If a cycle structure includes a cycle of length greater than 2, then there are an even number of permutations
with that cycle structure.
\label{lem:eulerian}
\end{lemma}

\begin{proof} Consider $P_r$, the set of permutations with a given cycle structure, $r$. We can  pair each $\sigma \in P_r$
 with its inverse $\sigma^{-1} \in P_r$, and so long as $\sigma \neq \sigma^{-1}$ for any $\sigma \in P_r$, $|P_r|$ will be even.  Suppose there exists a $\sigma \in P_r$ such that
  $\sigma = \sigma^{-1}$.  Then $\sigma^2 = e$, and so the order of $\sigma$ is at most 2.   
 The order of a permutation is the least common multiple of the orders of the elements of its cycle structure, so  $\sigma$ must not include a cycle of length greater than 2.  This is a contradiction; thus $|P_r|$ is even.
 \end{proof}

\begin{theorem}  For $n>3$ and $k<n$, $\Gamma_{k,n}$ is Eulerian if and only if $k$ is even or $k$ and $n$ are both odd.
\label{thm:eulerian}

\begin{proof} A graph is Eulerian if and only if it is connected and each vertex has an even degree.  In light of Theorem \ref{thm:connected} and the previously noted fact that
$\Gamma_{k,n}$ is $D_k(n)$-regular, in order to
ascertain if $\Gamma_{k,n}$ is Eulerian, we must determine whether
$D_k(n)$ is even or odd. 

If $k$ is even, we claim that $D_k(n)$ is the sum of even
numbers. Any cycle structure composed entirely of 2- or 1- cycles
will partition an even $k$, and thus any permutation which is in
$\mathcal{D}_{k,n}$ for an even $k$ will contain a cycle of length 3
or greater in its cycle decomposition. Now, $\mathcal{D}_{k,n} =
{P_{r_1}} \dot \cup P_{r_2} \cdots \dot \cup P_{r_m}$ such that no $r_i$
partitions $k$, and by Lemma \ref{lem:eulerian}, $|P_{r_i}|$
is even for all $i \in \{1,\ldots,m\}$. Thus, when $k$ is even,
$D_k(n)$ is even.

If $k$ and $n$ are both odd, again we see that every permutation in $\mathcal{D}_{k,n}$ will contain a cycle of length 3 or greater in its cycle decomposition, since an odd $k$ can be partitioned by a set of cycles of lengths 1 or 2 if there is at least one 1-cycle.  Furthermore, since $n$ is odd, there are no permutations whose cycle structure is composed only of length-2 cycles.  Thus, $D_k(n)$ is even.

Finally, we show that if $k$ is odd and $n$ is even, then $\Gamma_{k,n}$ is not Eulerian.  In this case,  $P_{\{2,2,\ldots,2\}}$ is in $CD_{ k,n}$.  By choosing pairs of elements for the cycles and dividing by the number of ways to order the cycles, we see that the number of permutations in $
P_{\{2,2,\ldots,2\}}$ is given by

$$\frac{\binom{n}{2} \binom{n-2}{2} \cdots \binom{2}{2}}{(\frac{n}{2}) !}= \frac{n(n-1)(n-2)\cdots(3)(2)(1)}{(2\cdot\frac{n}{2})(2\cdot(\frac{n}{2}-1))\cdots(6)(4)(2)} =$$

$$ \frac{n(n-1)(n-2)\cdots(3)(2)(1)}{n(n-2)\cdots(6)(4)(2)} = (n-1)(n-3)\cdots(5)(3)(1).$$

Since $n$ is even, the product $(n-1)(n-3)\cdots(5)(3)(1)$ is odd.  
Every other $k$-derangement in $S_n$ will contain a cycle
with length greater than 2, since any combination of 1-cycles or 1-
and 2-cycles will partition $k$.  So $D_k(n)$ is the sum of one
odd number and even numbers, and so is odd.  
\end{proof}
\end{theorem}

\section{Chromatic, independence and clique numbers for  $k=2$ and $n$ an odd prime power}

For the majority of this section, we will think of permutations in terms of the result of their application to the ordering $\{1,2,3,\ldots, n\}$.  Thus, $\{2,3,1,4,5\}$ represents the permutation which has moved 2 to the first position, 3 to the second, 1 to the third, and left 4 and 5 fixed; that is, the permutation $(132)(4)(5)$ in cycle notation, or the inverse of the permutation $\binom{12345}{23145}$, in two line notation.

We note that in order for $vu^{-1} $ (or equivalently, $v^{-1}u$) to be a $k$-derangement, it is necessary and 
	sufficient that no unordered $k$-tuple of elements be sent to the same unordered 
	$k$-tuple of positions by both $u$ and $v$.  For example, the permutation $u=\{2,3,1,4,5\}$ and $v=\{4,1,3,5,2\}$ both send the pair $\{1,3\}$ to the second and third positions.  Thus $(vu^{-1})_{(2)}(\{2,3\}) = \{2,3\}$, and so $vu^{-1}$ is not a 2-derangement and there is no edge between $u$ and $v$ in the 2-derangement graph.  
	More formally, suppose $u$ and $v$ both send the $k$-tuple  $M'=\{a_1',a_2', \ldots, a_k'\}$ to positions $M=\{a_1 ,a_2, \ldots, a_k \}$.  Then, $(vu^{-1})_{(k)}(M) = v_{(k)}(M') = M$.  Thus, $vu^{-1}$ is not a $k$-derangement.

On the other hand, if $u$ and $v$ send no $k$-tuple to the same positions we claim $vu^{-1}$ is a $k$-derangement.  Consider an arbitrary $k$-tuple, $M=\{a_1,a_2, \ldots, a_k\}$, and suppose $u$ maps the $k$-tuple $M'=\{a_1',a_2', \ldots, a_k'\}$ to the positions given in $M$.  Then $(vu^{-1})_{(k)}(M) =  v_{(k)}(M') \neq M$ since $v$ cannot send the $k$-tuple $M'$ to the same positions as $u$ does.  
Thus, $vu^{-1}$ is a $k$-derangement.

In Theorem \ref{thm:clique}, we find the clique number of the 2-derangement graph, $\omega(\Gamma_{2,n})$, for $n$ an odd prime power, by constructing a clique of maximal size.  
Before establishing this clique number, we note an upper bound on the clique number of a general $k$-derangement graph.

\begin{lemma}
For $k<n$,
$\omega(\Gamma_{k,n})\le\binom{n}{k}$.
\label{lem:clique}
\end{lemma}

\begin{proof}
The clique number of the $k$-derangement graph, $\omega(\Gamma_{k,n})$ cannot be greater than
$\binom{n}{k}$, since there are only $\binom{n}{k}$ subsets of size
$k$ and hence at most $\binom nk$ different unordered $k$-tuples of positions for an arbitrary $k$-tuple of elements to be sent under a permutation.
\end{proof}

\begin{theorem} If $n$ is an odd prime power, then $\omega(\Gamma_{2,n})=\binom{n}{2}$.
\label{thm:clique}
\end{theorem}

\begin{proof}
We will explicitly construct a clique with $\binom{n}{2}$ elements. Let $n=p^k$ with $p$ a prime greater than 2, and let $\mathbb{F}_{p^k}$ denote the field with $p^k$ elements. Rather than letting $S_n$ act on $[n]$, we will let it act on $\mathbb{F}_{p^k}$ and construct $\Gamma_{2,n}$ accordingly. Let $v=(x_1,\ldots,x_n)$ be an ordered $n$-tuple whose entries are the elements of $\mathbb{F}_{p^k}$ in some order. Given any function $\phi:\mathbb{F}_{p^k}\rightarrow\mathbb{F}_{p^k}$, we define $\phi(v)=(\phi(x_1),\ldots,\phi(x_n))$.  Partition the non-zero elements of $\mathbb{F}_{p^k}$ by pairing each element with its (additive) inverse, and let $T$ be a set obtained by choosing exactly one element from each pair, giving $|T|=(p^k - 1)/2$.  

Define $f_{s,\alpha}(x)=sx+\alpha$, and consider the set $X=\lbrace f_{s,\alpha}(v)|s\in T\mbox{ and }  \alpha\in\mathbb{F}_{p^k}\rbrace$. Since $s\neq 0$, $f_{s,\alpha}$ is a bijection and $f_{s,\alpha}(v)$ is a permutation of the elements of $\mathbb{F}_{p^k}$. We claim that $X$ is a clique in $\Gamma_{2,n}$. Suppose not; that is, suppose there are $s,t\in T$ and $\alpha,\beta\in\mathbb{F}_{p^k}$, $(s,\alpha)\neq(t,\beta)$, such that $f_{s,\alpha}(v)$ is not a 2-derangement of $f_{s,\beta}(v)$. In that case there exist $x,y\in\mathbb{F}_{p^k}$, $x\neq y$, such that either $f_{s,\alpha}(x)=f_{t,\beta}(x)$ and $f_{s,\alpha}(y)=f_{t,\beta}(y)$ or $f_{s,\alpha}(x)=f_{t,\beta}(y)$ and $f_{s,\alpha}(y)=f_{t,\beta}(x)$.  In the first case, subtracting the two equations and rewriting yields $(s-t)(x-y)=0$. If $s=t$, then $\alpha=\beta$ giving a contradition. If $s\neq t$, then $x=y$ and again we have a contradiction.  In the second case, subtracting and rewriting yields $(s+t)(x-y)=0$ and since $s+t\neq 0$ for $s,t\in T$, $x=y$ and this also give a contradiction.  Thus, $X$ is a clique of size $p^k ((p^k-1)/2) = \binom n2$.
\end{proof}

The next example illustrates the construction when $n=7$.  
\begin{example}
We build a clique of size $\binom 7 2$ in the derangement graph $\Gamma_{2,7}$ consisting of $\frac{7-1}{2}$ blocks, each of which contains 7 permutations.
We let $v = (1,2,3,4,5,6,7)$ (writing 7 instead of 0) and take $T = \{1,4,5\}$. Then$ f_{1,0}(v) = (1,2,3,4,5,6,7)$, $f_{4,0} (v)= (4,1,5,2,6,3,7)$, and $f_{5,0} (v)= (5,3,1,6,4,2,7)$. 
Increasing $\alpha$ from 0 cyclically permutes the 7-tuples.
Block 1 consists of the arrangements $\{f_{1,\alpha}(v)| \alpha\in \mathbb{F}_7\}$, that is the arrangement
 $(1, 2, 3, 4, 5, 6, 7)$ and the remaining 6 rotations of this arrangement (e.g., $(2,3,4,5,6,7,1)$, $(3,4,5,6,7,1,2)$, etc.).  Block 2 consists of the arrangement $f_{4,0}(v)$ along with all of its rotations.  Finally, Block 3 consists of $f_{5,0}(v)$ and its rotations.  
To see that these permutations form a clique, consider, for example, the pair $\{1,2\}$.  These elements are one position apart in block 1, two positions apart in block 2 and three positions apart in block 3 (counting the shortest distance between them either forwards or backwards).  So the pair $\{1,2\}$ cannot occupy the same positions in two permutations which appear in different blocks.  Furthermore, within a block, the rotations insure that the pair never occupies the same positions.  
\label{ex:clique}
\end{example}

Next we turn to the independence number $\alpha(\Gamma_{k,n})$ and the chromatic number $\chi(\Gamma_{2,n})$ of the $k$-derangement graph.  We will require the following lemma 
which has been adapted from Frankl and Deza's lemma \cite{FD} and applied to $k$-tuples of elements.

\begin{lemma}
For $k<n$, $\alpha(\Gamma_{k,n})\omega(\Gamma_{k,n})\le n!$
\label{lem:UpperBound}
\end{lemma}

\begin{proof}
Let $\sP$ be a set of permutations in $S_n$, every pair of which has at least one unordered $k$-tuple of elements in the same unordered $k$-tuple of positions.   That is, for any $u,v \in \sP$, there exists a set $M=\{a_1, \ldots,a_k\} \subseteq [n]$ such that $\left(v^{-1}u\right)_{(k)}(M) = M$.
Note that  $\sP$ is an independent set in the $k$-derangement graph.  
Let $\sQ$ be a set of permutations in $S_n$ such that each pair of permutations has no $k$-tuple of elements in the same positions; that is, $\sQ$ is a clique in the $k$-derangement graph.  We claim that products of the form $PQ$ with $P \in \sP$ and $Q \in \sQ$ give distinct permutations of $n$.  Suppose, for the sake of contradiction, that $P_1 Q_1 = P_2Q_2$ for $P_1, P_2 \in \sP$ and $Q_1, Q_2 \in \sQ$ with $P_1 \neq P_2$ and $Q_1 \neq Q_2$.  This implies that $P_1^{-1}P_2 = Q_1Q_2^{-1}$.  Now, since $P_1$ and $P_2$ are in $\sP$, there is a $k$-tuple of elements $M=\{a_1, \ldots,a_k\}$ such that $\left(P^{-1}_1P_2\right)_{(k)}(M) = M$.
However, this implies $\left(Q_1Q_2^{-1}\right)_{(k)}(M) = M$.  But we know that the permutations in $\sQ$ agree on no $k$-tuples, and so we must have $Q_1 =Q_2$ and hence, $P_1 = P_2$.    Finally, since each product gives a unique permutation of $n$, there can be no more than $n!$ such products. 
\end{proof}

\begin{theorem}
\label{thm:AlphaChi}
For $k<n$, $\alpha(\Gamma_{k,n})\ge k!(n-k)!$ and $\chi(\Gamma_{k,n})\le\binom{n}{k}$.
\end{theorem}
\begin{proof}
Consider $H$, the set of all permutations in $S_n$ that send $ \lbrace 1,2,\ldots,k\rbrace$ to itself (and hence $\lbrace k+1,\ldots,n\rbrace$ to itself). It is clear that $H$ is a subgroup of $S_n$ isomorphic to $S_k\times S_{n-k}$ and that $|H|=k!(n-k)!$. Since the unordered $k$-tuple $ \lbrace 1,2,\ldots,k\rbrace$ is fixed, none of these are $k$-derangements of each other, so $H$ is an independent set and $\alpha(\Gamma_{k,n})\ge k!(n-k)!$

The cosets of $H$ partition $S_n$, and each forms an independent set, since $\tau_1,\tau_2\in\sigma H$ implies that $\tau_1^{-1}\tau_2\in H$ is not a $k$-derangement and hence the vertices associated to $\tau_1$ and $\tau_2$ are not connected by an edge.  Giving each of the $\frac{n!}{k!(n-k)!}=\binom{n}{k}$ cosets a different color results in a valid coloring of $\Gamma_{k,n}$, so $\chi(\Gamma_{k,n})\le\binom{n}{k}$.
\end{proof}

\begin{corollary}
For $n$ an odd prime power, $\alpha(\Gamma_{2,n})=2(n-k)!$ and $\chi(\Gamma_{2,n})=\binom{n}{2}$.
\end{corollary}
\begin{proof}
By Lemma \ref{lem:UpperBound} and Theorem \ref{thm:clique}, we have $\binom{n}{2}\cdot\alpha(\Gamma_{2,n})\le n!$. Thus $\alpha(\Gamma_{2,n})\le n!\cdot\frac{2(n-2)!}{n!}=2(n-2)!$ and Theorem \ref{thm:AlphaChi} gives the reverse inequality.  For any graph $G$, $\chi(G)\ge\omega(G)$, so by Theorem \ref{thm:clique}, $\chi(\Gamma_{2,n})\ge\binom{n}{2}$ and again Theorem \ref{thm:AlphaChi} gives the reverse inequality.
\end{proof}

\section{Further Questions}

While the last section focused on properties of the 2-derangement graphs for $n$ an odd prime power, we are interested in finding formulas for $\omega(\Gamma_{k,n})$, $ \alpha(\Gamma_{k,n})$ and $ \chi(\Gamma_{k,n})$ for arbitrary $k$ and $n$.  
We have some faint hope that the bounds given in Theorem \ref{thm:AlphaChi} are actually equalities, but those of Theorem \ref{thm:clique} cannot be, since $\omega(\Gamma_{2,4})=5<\binom{4}{2}$ (via a computer search).
When $n$ is not an odd prime power, the clique construction of Theorem \ref{thm:clique} fails to work. 
If $n$ is not a prime power, then there is no field of that cardinality, and if $n=2^k$, then the condition that $\alpha+\beta\neq 0$ fails since $\alpha+\alpha = 0$ for all $\alpha \in\mathbb{F}_{2^k}$.
In fact, we believe the clique number for $n$ not an odd prime power is strictly smaller than $\binom{n}{k}$.  We found a clique of size 9 for $n=6$ and $k=2$, and so  $ 9 \leq \omega(\Gamma_{2,6})\leq 15$.

In another direction, the numerical evidence is overwhelming that the derangement graphs are Hamiltonian.  We hope to explore this and other questions in future work.

\section*{Acknowledgements} The first author worked on this topic with the third 
author at an REU at Missouri State University in the summer of 2009
(award \#0552573). The authors would like to acknowledge Sam Tencer's
contribution to this investigation.

\bibliographystyle{amsalpha}

\end{document}